\documentclass[twoside]{mfa}

\usepackage{amssymb,amsfonts}
\renewcommand\volume{{\bf{...}} (...)}        %% volume
\usepackage{graphicx}

\newcommand\N{\mathbb{N}}

\newcommand\R{\mathbb{R}}

\newcommand\sgn{\operatorname{sgn}}
\newcommand\NSV{{\mathcal{NSV}}}
\newcommand\SV{{\mathcal{SV}}}
\newcommand\NRV{{\mathcal{NRV}}}
\newcommand\RV{{\mathcal{RV}}}
\newcommand\DS{{\mathcal{DS}}}
\newcommand\IS{{\mathcal{IS}}}
\newcommand\mS{{\mathcal{S}}}
\newcommand\drm{{\mathrm{d}}}

\begin{document}
%%============================================================================%%
%%                              title page                                    %%
%%============================================================================%%

\title[Increasing solutions of half-linear delay DE\lowercase{s}]{On increasing solutions of half-linear delay differential equations}

\author[S. Matucci]{Serena Matucci}
\address{Department of Mathematics and Informatics ``U. Dini'',
	University of Florence, I-50139 Florence, Italy}
\email{serena.matucci@unifi.it}
%\curraddr{}
%\urladdress{http://www.anonymous.com/anonymous/index.html}

\author[P. \v Reh\'ak]{Pavel \v Reh\'ak}
\address{Institute of Mathematics,
	FME,
	Brno University of Technology,
	Technick\'a 2, Brno, CZ-61669,
	Czech Republic}
\email{rehak.pavel@fme.vutbr.cz}

%\dedicatory{Dedicated to ...}
\thanks{The research of the second author has been supported by the grant GA20-11846S of the Czech Science Foundation.  The first author has been partially supported by Gnampa, National Institute for Advanced Mathematics (INdAM)}

%%%%%%%% Keywords %%%%%%%%%%%%%%%%%%%%%%%%%%%%%%%

%%%%%%%% AMS classification 2010 %%%%%%%%%%%%%%%%%%%%%%%%
{26A12}      %% \subjclass{primary}{secondary}
%%%%%%%% Abstract %%%%%%%%%%%%%%%%%%%%%%%%%%%%%%%%%%%%
\begin{abstract}
We establish conditions guaranteeing that all eventually positive increasing
solutions of
%the equation $(r(t)\Phi(y'))'=p(t)\Phi(y(\tau(t)))$, where $\Phi(u)=|u|^{\alpha-1}\sgn u$, $\alpha>1$ and $\tau(t)\le t$,
a half-linear delay differential equation
are regularly varying and derive precise asymptotic formulae for them.
The results here presented are new also in the linear case and some of the observations are original
also for non-functional equations.
A substantial difference between the delayed and non-delayed case for eventually positive decreasing solutions is pointed out.

\end{abstract}

%%%%%%%% Environments Theorem, Definition, etc. %%%%%%%%%%%%%%%%%%%%%%%%%%%%%%%%%%%
\newtheorem{theorem}{Theorem}[section]
\newtheorem{corollary}[theorem]{Corollary}
\newtheorem{lemma}[theorem]{Lemma}
\newtheorem{proposition}[theorem]{Proposition}

\theoremstyle{definition}
\newtheorem{definition}[theorem]{Definition}
\newtheorem{problem}[theorem]{Problem}
\newtheorem{example}[theorem]{Example}
\newtheorem{remark}[theorem]{Remark}

\numberwithin{equation}{section}

%%%%%%% Ttitlepage typesetting %%%%%%%%%%%%%%%%%%%%%%%%%%%%%%%%%%%%%%%%%
\maketitle

\medskip\centerline{paper submitted to \emph{Mathematics for applications}}

%%======================================================================%%
%%                             Article body                             %%
%%======================================================================%%
\section{Introduction}

Let us consider the equation
\begin{equation} \label{e}
(r(t)\Phi(y'))'=p(t)\Phi(y(\tau(t)))
\end{equation}
where $r,p$ are positive continuous functions on $[a,\infty)$ and $\Phi(u)=|u|^{\alpha-1}\sgn u$ with $\alpha>1$.
Throughout the paper we assume (unless not stated otherwise)
that $\tau$ satisfies the following conditions:
\begin{equation} \label{tau}
\tau\in C^1,\ \tau'>0,\ \tau(t)\le t,
\end{equation}
and
\begin{equation} \label{tau+}
\limsup_{t\to\infty}\frac{t}{\tau(t)}<\infty.
\end{equation}
The above conditions imply $\lim_{t\to\infty}\tau(t)=\infty$,
% In some of the theorems we will need additional conditions on $\tau$.
 and are fulfilled by standard examples of delay, e.g., $\tau(t)=t-\sigma$ with $\sigma>0$, or $\tau(t)=\lambda t$ with $\lambda\in(0,1)$.
Note that, in
contrast to the linear case where an equation with a general
delay can be transformed into an equation with a constant delay (see \cite{n:90}),
in the half-linear case it makes a good sense to consider a
general $\tau$. Solutions of \eqref{e} are understood in the classical sense, i.e., a solution $y$ of \eqref{e} is a $C^1$ function defined in an interval $I \subseteq [a, \infty)$,
such that $r\Phi(y')\in C^1(I)$ and $y$ satisfies \eqref{e} in $I$.%an interval under consideration.

Our aim is to describe asymptotic behavior of all eventually positive increasing solutions of \eqref{e} via the theory of regular variation.
Regularly varying functions have been shown to be a very useful tool in studying asymptotic properties of various type of differential equations, see,
in particular, the monograph \cite{maric}. Linear and half-linear ordinary differential equations have been considered in the framework
of regular variation e.g. in \cite{book, geluk1990,grimm-hall,jaros-kusano,jar-kus-tan2003,pr-methods,pr-asymp,rt}.
Linear and half-linear functional differential equations have been studied in this framework in
	 \cite{grimm-hall,kusano-maric2006,kusano-maric2007georg,kusano-maric2007bull,manojl-tanig,pituk,tanigawa2008,tanigawa2012}.
The typical result in the works \cite{kusano-maric2006,kusano-maric2007georg,kusano-maric2007bull,manojl-tanig,tanigawa2008,tanigawa2012}
is the existence of a regularly varying solution, obtained by means of a topological approach, as, for example, the Schauder-Tychonoff fixed point theorem.
%typically  based on the following idea:
%To form a family of associated differential equations without functional argument each of which possesses, say,
%a slowly varying solution, and then to
%look for, with the help of the Schauder-Tychonoff fixed point theorem, the one from the family which gives birth to the desired solution of the original functional differential equation.
Our approach is different. We deal with the entire  class of eventually positive increasing solutions (which we assume to be non empty),
and we show that it is a subset of the class of regularly varying solutions. In addition, we derive precise asymptotic formula for each solution.
All our results are new also in the (functional) linear case, and
some of them are new in the non-functional (half-linear as well as linear) case. The comparison with existing results is discussed in remarks after the main results.

The paper is organized as follows. In the next section we give some basic information on equation \eqref{e} and recall selected
facts from the theory of regular variation,  which will include also De Haan class $\Pi$.
The main results are presented in Section~\ref{S: main}. In addition of showing regular variation of increasing solutions and deriving asymptotic formulae,
we briefly discuss also decreasing solutions; in particular, we show that
under reasonable assumptions they may exhibit a quite different behavior when compared with the solutions of non-delayed equations.
Directions for a possible future research are discusses as well. The last section contains all the proofs.

\section{Preliminaries}

As usual, the relation $f(t)\sim g(t)$ as $t\to\infty$ means $\lim_{t\to\infty}f(t)/g(t)=1$, the relation
$f(t)\asymp g(t)$ as $t\to\infty$ means that $\exists c_1,c_2\in(0,\infty)$ such that $c_1 g(t)\le f(t)\le c_2 g(t)$
for large $t$, and $f(t)=o(g(t))$ as $t\to\infty$ means that $\lim_{t\to\infty}f(t)/g(t)=0$.

As for nonoscillatory solutions of \eqref{e} (i.e., the solutions which are eventually of one sign), without loss of generality, we work
just with positive solutions, i.e., with the class
$$ \mathcal{S} = \{ y: y(t) \text{ is a positive solution of } \eqref{e} \text{ for large }t\}. $$
We wish to include our results into the framework of a standard classification of nonoscillatory solutions, which is given in
what follows. Because of the sign conditions on the coefficients,
all positive solutions of \eqref{e} are eventually monotone,
therefore any such a solution belongs to one of the following
disjoint classes:
\begin{equation*}
\begin{aligned}
\IS = \left\{ y \in \mathcal{S}: y'(t)>0 \text{ for large } t \right\},\ \
\DS = \left\{ y \in \mathcal{S}: y'(t)<0 \text{ for large } t \right\}.
\end{aligned}
\end{equation*}
As for the nonemptiness of these classes as well as of the subclasses defined below,
general conditions  are not known, as far as we know. Partial results related to the problem of the existence of nonoscillatory solutions were obtained in  \cite{kusano-lalli}.
Further, as a by-product of the investigation of half-linear delay differential equations in the framework of regular variation via a fixed point approach, some existence theorems under a setting which is close to our ones can be found e.g. in
 \cite{kusano-maric2007bull,tanigawa2008}.

We shall focus on the class $\IS$.
This class can be divided
into the mutually disjoint subclasses:
\begin{equation*}
\begin{aligned}
&\IS_{\infty} = \left\{ y \in \IS: \lim_{t \to \infty} y(t) = \infty \right\}, \ \ \IS_B = \left\{ y \in \IS: \lim_{t \to \infty} y(t) = b \in \mathbb R \right\}, \\
%&\DS_B = \left\{ y \in \DS: \lim_{t \to \infty} y(t) = b > 0 \right\}, \ \ \DS_0 = \left\{ y \in \DS: \lim_{t \to \infty} y(t) = 0 \right\}.
\end{aligned}
\end{equation*}
Define the so-called quasiderivative $y^{[1]}$ of
$y\in\mathcal{S}$ by $y^{[1]}=r\Phi(y')$. We introduce the
following convention
$$
\begin{aligned}
&\IS_{u,v}=\left\{ y \in \IS: \lim_{t \to \infty} y(t) =u,\ \ \lim_{t\to\infty}y^{[1]}(t)=v \right\}.\\
%&\DS_{u,v}=\left\{ y \in \DS: \lim_{t \to \infty} y(t) =u,\ \ \lim_{t\to\infty}y^{[1]}(t)=v \right\},
\end{aligned}
$$
where $u,v \in \{B, \infty\}$. If $u=B$ [$v=B$] we
mean that the value of the corresponding limit is a real nonzero number.
Using this convention we further distinguish the following types
of solutions which form subclasses in $\IS_B$ and
$\IS_\infty$ (we list only those ones that are not a-priori
excluded):
\begin{equation} \label{sbc}
%\DS_{0,0}, \DS_{0,B}, \DS_{B,0}, \DS_{B,B},
\IS_{B,B},
\IS_{B,\infty}, \IS_{\infty,B}, \IS_{\infty,\infty}.
\end{equation}
In some places we need to emphasize that the classes and
subclasses of eventually positive  increasing
solutions are associated with a particular equation, say equation ($\ast$);
then we write
\begin{equation} \label{hvezda}
\text{$\IS^{(\ast)},\ \IS_B^{(\ast)}$, etc.}
\end{equation}
\medskip

In the second part of this section we recall basic information on the Karamata theory
of regularly varying functions and the de Haan class $\Pi$; for more information see the monographs
\cite{bgt,geluk,dehaan}.

A measurable function $f:[a,\infty)\to(0,\infty)$ is called \textit{regularly varying (at infinity) of index}
$\vartheta$, $\vartheta\in\R$, if
\begin{equation} \label{D:RV}
\lim_{t\to\infty}{f(\lambda t)}/{f(t)}=\lambda^\vartheta
\ \ \text{ for every $\lambda>0$;}
\end{equation}
we write $f\in\RV(\vartheta)$. If $\vartheta=0$, then we speak about \textit{slowly varying}
functions;
we write $f\in\SV$, thus $\SV=\RV(0)$. By $\RV$ we mean either abbreviation of ``regularly varying'' or
$\RV=\bigcup_{\vartheta\in\R}\RV(\vartheta)$.

A function $ f \in \RV(\vartheta) $ if and only if there
exists a function $ L \in \SV $ such that $ f(t) = t^{\vartheta}
L(t) $ for every $ t. $ The slowly varying component of $f\in
\RV(\vartheta)$ will be denoted by $L_f$, i.e.,
$L_f(t):={f(t)}/{t^{\vartheta}}.$

The following result (the so-called Uniform Convergence Theorem, see e.g. \cite{bgt})
is one of the most fundamental theorems in the theory. Many
important properties of $\RV$ functions follow from it.
\begin{proposition}	
If $f\in\RV(\vartheta)$, then the relation \eqref{D:RV} holds uniformly on each compact $\lambda$-set in
	$(0,\infty)$.
\end{proposition}	
Another important result in the theory of  $\RV$ functions is the  Representation Theorem (see e.g. \cite{bgt}).
\begin{proposition}	
$f\in \RV(\vartheta)$ if and only if
\begin{equation} \label{repr}
f(t)=\varphi(t)t^\vartheta\exp\left\{\int_{t_0}^t\frac{\psi(s)}{s}\, \drm s\right\},
\end{equation}
$t\ge t_0$, for some $t_0>0$, where $\varphi, \psi$ are measurable with $\lim_{t\to\infty}\varphi(t)=C\in(0,\infty)$ and $\lim_{t\to\infty}\psi(t)=0$.
A function $f\in \RV(\vartheta)$ can alternatively be represented as
$$ %$\begin{equation} \label{repr2}
f(t)=\varphi(t)\exp\left\{\int_{t_0}^t\frac{\omega(s)}{s}\, \drm s\right\},
$$ %$\end{equation}
$t\ge t_0$, for some $t_0>0$, where $\varphi, \omega$ are measurable
with $\lim_{t\to\infty}\varphi(t)=C\in(0,\infty)$ and
$\lim_{t\to\infty}\omega(t)=\vartheta$.
\end{proposition}	

A regularly varying function $ f $ is said to be \textit{normalized regularly varying}, and  we write $ f \in \NRV(\vartheta), $ if $\varphi(t)\equiv C$ in
\eqref{repr} or in the alternative representation. If \eqref{repr} holds with $ \vartheta = 0$ and $\varphi(t)\equiv C$, we say that $ f $ is \textit{normalized slowly varying}, and we write $ f
\in \NSV. $ Clearly, if $ f $ is a $ C^1 $ function and $ \lim_{t \to \infty} {tf'(t)}/{f(t)} = \vartheta, $ then $ f \in \NRV (\vartheta). $
Conversely, if $f\in\NRV(\vartheta)\cap C^1$, then $\lim_{t\to\infty}tf'(t)/f(t)=\vartheta$.

The classes of regularly varying solutions of \eqref{e} are defined as follows:
$$
\begin{aligned}
&\mS_{\SV}=\mS\cap\SV, \quad
\mS_{\RV}(\vartheta)=\mS\cap\RV(\vartheta),\\
&\mS_{\NSV}=\mS\cap\NSV, \quad
\mS_{\NRV}(\vartheta)=\mS\cap\NRV(\vartheta).
\end{aligned}
$$

The  Karamata Integration Theorem (see e.g.
\cite{bgt,geluk}) plays a very important role in our theory. Its statement can be summarized as follows.

\begin{proposition} \label{T:karam}
	Let $ L \in \SV.$
	\begin{itemize} \itemsep=1mm
	\item[(i)]
	If $\vartheta<-1$, then
	$
	\int_t^{\infty} s^{\vartheta}L(s)\, \drm s  \sim  t^{\vartheta + 1} L(t)/(-\vartheta - 1)
	$
	as $t\to\infty$.
	
\item[(ii)]
	If $\vartheta>-1$, then
	$
	\int_a^t s^{\vartheta}L(s)\,\drm s  \sim
	t^{\vartheta + 1} L(t)/(\vartheta + 1)$
	as $t\to\infty$.
	
	\item[(iii)]
	If $ \int_a^{\infty} L(s)/s\, \drm s $ converges, then $ \widetilde L(t) = \int_t^{\infty} L(s)/s\, \drm s$ is a $\SV$ function; if $ \int_a^{\infty} L(s)/s\, \drm s $ diverges, then
	$ \widetilde L(t) = \int_a^t L(s)/s\, \drm s$ is a $\SV$ function; in both cases, $ {L(t)}/ {\widetilde L(t)} \to 0$ as $t\to\infty$.
\end{itemize}
\end{proposition}

Here are further useful properties of $\RV$ functions.
	The proofs of (i)--(viii) are either easy or can be found in
	\cite{bgt,geluk}. The proof of (ix) can be found in the last section.
	
\begin{proposition}\label{P:RV}\
	\begin{itemize} \itemsep=1mm
		\item[(i)]
		If $f\in\RV(\vartheta)$, then $\ln f(t)/\ln t\to\vartheta$ as
		$t\to\infty$. It then clearly implies that
		$\lim_{t\to\infty}f(t)=0$ provided $\vartheta<0$, or
		$\lim_{t\to\infty}f(t)=\infty$ provided $\vartheta>0$.
		
		\item[(ii)]
		If $f\in\RV(\vartheta)$, then $f^\alpha\in\RV(\alpha\vartheta)$ for every $\alpha\in\R$.
		
		\item[(iii)]
		If $f_i\in\RV(\vartheta_i)$, $i=1,2$, $f_2(t)\to\infty$ as $t\to\infty$, then $f_1\circ f_2\in\RV(\vartheta_1\vartheta_2)$.
		
		\item[(iv)]
		If $f_i\in\RV(\vartheta_i)$, $i=1,2$, then $f_1+ f_2\in\RV(\max\{\vartheta_1, \vartheta_2\})$.
		
		\item[(v)]
		If $f_i\in\RV(\vartheta_i)$, $i=1,2$, then $f_1 f_2\in\RV(\vartheta_1+\vartheta_2)$.
		
		\item[(vi)]
		If $f_1,\dots,f_n\in\RV$, $n\in\N$, and $R(x_1,\dots,x_n)$ is a rational function with nonnegative
		coefficients, then $R(f_1,\dots,f_n)\in\RV$.
		
		\item[(vii)]
		If $L\in\SV$ and $\vartheta>0$, then $t^\vartheta L(t)\to\infty,$ $t^{-\vartheta}L(t)\to 0$
		as $t\to\infty$.
		
		\item[(viii)]
		If $f\in\RV(\vartheta)$, $\vartheta\ne 0$, then there exists $g\in
		C^1$ with $g(t)\sim f(t)$ as $t\to\infty$ and such that
		$tg'(t)/g(t)\to\vartheta$, whence $g\in\NRV(\vartheta)$. Moreover,
		$g$ can be taken such that $|g'|\in\NRV(\vartheta-1)$.
		
		\item[(ix)] % puvodne lemma \label{L:ffnrv}
		If $|f'|\in\RV(\vartheta)$, $\vartheta\in\R$, with $f'$ being eventually of one sign
		and $f(t)\to 0$ or $f(t)\to\infty$ as $t\to\infty$, then
		$f\in\NRV(\vartheta+1)$.
	\end{itemize}
\end{proposition}

Some other properties of $\RV$ functions, needed in the proofs of the main results,  are presented in some auxiliary lemmas  in the last section.		

\medskip	

Finally, we recall the definition of the De Haan class $\Pi$, together with some useful properties.

A measurable function $f:[a,\infty)\to\R$ is said to belong to the
class $\Pi$ if there exists a function $w:(0,\infty)\to(0,\infty)$
such that for $\lambda>0$
\begin{equation} \label{D:dehaan1}
\lim_{t\to\infty}\frac{(f(\lambda t)-f(t))}{w(t)}=\ln \lambda;
\end{equation}
we write $f\in\Pi$ or $f\in\Pi(w)$. The function $w$ is called an
\textit{auxiliary function} for $f$. The class $\Pi$, after taking
absolute values, forms a proper subclass of $\SV$.
	
Next we give some properties of the class $\Pi$.
	The proofs of (i)--(ii) can be found in the monographs \cite{geluk, dehaan}.
	For
	(iii) see e.g. \cite{pr-methods}.

\begin{proposition}\label{P:Pi}
	\
	\begin{itemize} \itemsep=0mm
		\item[(i)]
		If $f\in\Pi(v),$ then
		$ %$\begin{equation} \label{fv}
		v(t)\sim f(t)-\frac 1t\int_a^t f(s)\,\drm s
		$ %$\end{equation}
		as $t\to\infty$.
		
		\item[(ii)] % puvodne lemma \label{L:PiSV}
		If $f\in\Pi$, then $\lim_{t\to\infty}f(t)=:f(\infty)\le\infty$
		exists. If the limit is infinite, then $f\in\SV$. If the limit is
		finite, then $f(\infty)-f(t)\in\SV$.
		
		\item[(iii)] % puvodne lemma \label{L:fPi}
		If $f'\in\RV(-1)$, then $f\in\Pi(tf'(t))$.
		\end{itemize}
		\end{proposition}

		\section{Main results} \label{S: main}
		
		Denote
		$$
		G(t)=\Phi^{-1}\left(\frac{tp(t)}{r(t)}\right)= \left(\frac{tp(t)}{r(t)}\right)^{\beta-1},
		$$
		where $\beta$ is the conjugate number of $\alpha$, i.e., $1/\alpha+1/\beta=1$. We start by establishing conditions guaranteeing slow variation of increasing solutions, for which we also derive asymptotic formulae.
		
		\begin{theorem} \label{T:SV}
		(I) Assume that $\int_a^\infty p(s)\,\drm s=\infty$ and
		\begin{equation} \label{trintp}
		\lim_{t \to \infty}\frac{t^{\alpha-1}}{r(t)}\int_a^t p(s)\,\drm s=0.
		\end{equation}	
		Then $\IS\subset\NSV$.
		If, in addition, $p\in\RV(\delta)$ with
		$\delta>-1$, then for any $y\in\IS$  the
		following hold.
		
		(i) If $\int_a^\infty G(s)\,\drm s=\infty$, then $y$ satisfies the formula
		\begin{equation} \label{f1}
		y(t)=\exp\left\{\int_a^t(1+o(1))\frac{G(s)}{\Phi^{-1}(\delta+1)}\,\drm s\right\}
		\end{equation}	
		as $t\to\infty$. Moreover, $\mS_{\NSV}=\mS_{\SV}=\IS=\IS_{\infty,\infty}$.
		
		(ii)  If $\int_a^\infty G(s)\,\drm s<\infty$,
		then $y$ satisfies the formula
		\begin{equation} \label{f2}
		y(t)=N\exp\left\{-\int_t^\infty(1+o(1))\frac{G(s)}{\Phi^{-1}(\delta+1)}\,\drm s\right\}
		\end{equation}
		as $t\to\infty$, where $N:=\lim_{t\to\infty}y(t)\in(0,\infty)$.
		Moreover, $\mS_{\NSV}=\mS_{\SV}=\IS=\IS_{B,\infty}$.
		
		(II) Let conditions at point (I) be satisfied and let $r\in\RV(\gamma)$ hold. Then $\gamma\ge\alpha+\delta$. For
		any $y\in\IS$ it holds
		$y(t)\in\Pi(ty'(t))$ provided $\gamma=\alpha+\delta$. Moreover, if  $\gamma=\alpha+\delta$ and $\int_a^\infty G(s)\,\drm s<\infty$, or  $\gamma>\alpha+\delta$, then
		\begin{equation} \label{N-y}
		N-y(t)\sim\frac{N}{\Phi^{-1}(\delta+1)}\int_t^\infty G(s)\,\drm s\in\RV((\delta+1-\gamma)(\beta-1)+1)
		\end{equation}
		as $t\to\infty$. In particular, if  $\gamma=\alpha+\delta$, then
		$|N-y|\in\SV$ and
		\begin{equation} \label{LLN}
		\frac{L_p^{\beta-1}(t)}{L_r^{\beta-1}(t)(N-y(t))}=o(1)
		\end{equation}
		as $t\to\infty$.
	\end{theorem}
	
	Condition \eqref{trintp} is necessary in a certain sense. More precisely, the following lemma holds.
	
	\begin{lemma} \label{L:nec}
		Let $r\in\RV(\gamma)$ with $\gamma>\alpha-1$, and $\tau(t)\asymp t$ as
		$t\to\infty$.
		If there exists $y\in\IS\cap\NSV$,
		then \eqref{trintp} holds.
		If, in addition, $p\in\RV(\delta)$, then
		\begin{equation} \label{tLpLr}
		\lim_{t \to \infty}t^{\alpha+\delta-\gamma}\frac{L_p(t)}{L_r(t)}=0.	
		\end{equation}	
	\end{lemma}

		Note that if
		\begin{equation} \label{pr}
		p\in\RV(\delta),\ \ r\in\RV(\delta+\alpha)
		\end{equation}
		hold, then
	\begin{equation} \label{G}
		G(t)=\frac{1}{t}\left(\frac{L_p(t)}{L_r(t)}\right)^{\beta-1}\in\RV(-1).
	\end{equation}
Observe that \eqref{pr} along with $\delta>-1$	yield $\gamma>\alpha-1$ (i.e., the assumption from Lemma~\ref{L:nec}).	
		Further, condition \eqref{trintp} (as well as condition \eqref{tLpLr}) reduces to
		\begin{equation} \label{Lpr}
		\lim_{t \to \infty}\frac{L_p(t)}{L_r(t)}=0,	
		\end{equation}
which is \eqref{tLpLr} with $\gamma=\delta+\alpha$.		

\medskip
	
		From the non-functional case (see \cite{rt}) we know that $\int_a^\infty r^{1-\beta}(s)\,\drm s=\infty$
		and
		\begin{equation} \label{trp1}
		\lim_{t \to \infty}\frac{t^{\alpha-1}}{r(t)}\int_t^\infty p(s)\,\drm s=0
		\end{equation}
		imply $\DS\subset\NSV$. Under the additional condition $p\in\RV(\delta)$ with $\delta<-1$, asymptotic formulae similar to \eqref{f1} and \eqref{f2} can be established.
		In particular, $\SV$ solutions must be sought among decreasing solutions.
		This result can be understood as a certain complement to Theorem~\ref{T:SV}. However, as we will see, here we encounter a quite big difference between the functional and non-functional case.
		For instance, let $r(t)=1$, $\alpha=2$, and $\tau(t)=t-1$ in \eqref{e}. Then $y(t)=e^{-t^2}$ is a solution of \eqref{e} which takes here the form $y''=p(t)y(t-1)$, where $p(t)=(4t^2-2)e^{1-2t}$. Clearly, $\lim_{t \to \infty}t^2p(t)=0$, and this condition implies also $\lim_{t\to\infty}t\int_t^\infty p(s)\,\drm s=0$, i.e., condition \eqref{trp1}. However,
		$y\in\DS$, but $y\not\in\RV$. Thus we see that (in contrast to the situation for $\IS$ discussed in Theorem~\ref{T:SV}), for decreasing solutions in our framework, qualitative behavior may substantially change when passing from ordinary to functional equations. An open problem is whether or not $\int_a^\infty r^{1-\beta}\,\drm s=\infty$ and \eqref{trp1} imply $\DS\subset\NSV$ when assuming, in addition, that $p\in\RV(\delta)$ with $\delta<-1$.
		
		Note that if $p\in\RV(\delta)$ with $\delta\ne-1$, then both \eqref{trintp} when $\delta>-1$, and \eqref{trp1} when $\delta<-1$, reduce to the condition
		$$
		\lim_{t\to\infty}\frac{t^\alpha p(t)}{r(t)}=0,
		$$
which can easily be seen from the Karamata theorem.		

\medskip

	We proceed with a complementary case with respect to Theorem~\ref{T:SV} in the sense
	that we study increasing solutions when $\delta<-1$
	under the conditions
	\eqref{pr}
	and \eqref{trintp} (which yield \eqref{Lpr}). We shall prove
	regular variation of these solutions where the index is equal to
\begin{equation} \label{def-rho}
	\varrho:=\frac{-1-\delta}{\alpha-1},
	\end{equation}
	and derive asymptotic formulae.
	Denote
	$$
	H_\tau(t)=(t\tau'(t))^{\alpha-1}\frac{p(t)}{r(\tau(t))}.
	$$
	If \eqref{pr} and the first condition in \eqref{tau-add} hold,
	then $H_\tau\in\RV(-1)$ and $H_\tau$ can be written as
	$$
	 H_\tau(t)=\frac1t\left(\frac{t}{\tau(t)}\right)^{\delta+\alpha}(\tau'(t))^{\alpha-1}\frac{L_p(t)}{L_r(\tau(t))}.
	$$
	%Note that the first condition in \eqref{tau-add} is implied by
	%$r\in\RV(\delta+\alpha)$ and $\tau'\in\SV$.
	%If $\tau'(t)\sim 1$,
	%then $H_\tau(t)\sim L_p(t)/(tL_r(t))$ as $t\to\infty$.
	Since the
	convergence/diver\-gen\-ce of the integrals $\int_a^\infty G$ and
	$\int_a^\infty H_\tau$ plays an important role, the following
	example is of interest.
	\begin{example} \label{E:HG}
		%\footnote{bylo by fajn ukazat, ze i $\tau$ muze mit vliv na konvergenci ci divergenci (i kdyz pochybuji),
		%anebo aspon ukazat, ze za nasich podminek muze NENASTAT tenhle pozadavek na $\tau$,
		%jinak bychom totiz mohli i v obecnem pripade asymptoticky zjednodusit $H_\tau$ tak, ze by nezaviselo na $\tau$.
		%--- i kdyz to je vlastne ok, protoze vzdy tam muzeme vtlacit aspon $\lambda$.
		%}
		Assume for simplicity that $\tau'(t)\sim \lambda$ with $\lambda\in(0,1]$ as $t\to\infty$; note that $\tau(t)=\lambda t$ or $\tau(t)=t-\sigma$
		both satisfy this condition.
		Then $$H_\tau(t)\sim\frac{\lambda^{-\delta-1} t^{\alpha-1}p(t)}{r(\tau(t))}\sim\frac{\lambda^{-\delta-1} t^{\alpha-1}p(t)}{r(t)}
		=\frac{\lambda^{-\delta-1}}{t}\cdot\frac{L_p(t)}{L_r(t)}$$ as $t\to\infty$
		by Lemma~\ref{L:fg} provided \eqref{pr} holds.
		Recall that $G$ takes here the form \eqref{G}.
		Taking $r,p$ such that
		$L_p(t)/L_r(t)\sim\ln^\gamma t$ as $t\to\infty$, $\gamma\in(-\infty,0)$, we see that
		\eqref{Lpr} is fulfilled. Moreover, if $\alpha<2$ and
		$-1<\gamma<1-\alpha$ then $\int_a^\infty G<\infty$ while
		$\int_a^\infty H_\tau=\infty$, and if $\alpha>2$ and $1-\alpha<\gamma<-1$
		then $\int_a^\infty G=\infty$ while $\int_a^\infty H_\tau<\infty$. Note
		that in the linear case (i.e., $\alpha=2$), $\lambda^{-\delta-1} G(t)\sim H_\tau(t)$ as $t\to\infty$,
		and thus half-linear equations exhibit more complex behavior.
	\end{example}
	
	\begin{theorem} \label{T:RV}
		Assume that \eqref{pr} holds and $\delta<-1$.
		%\eqref{pr} and
		%\eqref{Lpr} hold.
		Let
		\begin{equation} \label{tau-add}
		(r^{1-\beta}\circ\tau)\tau'\in\RV(\delta(1-\beta)-\beta)\ \ \text{and}\ \
		\left(\frac{L_p(t)}{L_r(t)}\right)^{\beta-1}\tau'(t)\to 0
		\end{equation}
		as $t\to\infty$. Then $\IS\subset\NRV(\varrho)$, where $\varrho$ is defined in \eqref{def-rho}, and
		$y^{[1]}(t)\in\Pi(tp(t)\Phi(y(\tau(t))))$ for every $y\in\IS$. In addition, for any $y\in\IS$ the following hold.
		
		(i) If $\int_a^\infty H_\tau(s)\,\drm s=\infty$, then
		\begin{equation} \label{f11}
		y(t)=tr^{1-\beta}(t)\exp\left\{\int_a^t(1+o(1))\frac{\beta-1}{\Phi(\varrho)}H_\tau(s)\,\drm s\right\}
		\end{equation}
		as $t\to\infty$. Moreover,  $\mS_{\NRV}(\varrho)=\mS_{\RV}(\varrho)=\IS=\IS_{\infty,\infty}$.
		
		(ii)  If $\int_a^\infty H_\tau(s)\,\drm s<\infty$, then
		\begin{equation} \label{f21}
		y(t)=A+\int_a^t M^{\beta-1}r^{1-\beta}(s)\exp\left\{-\int_s^\infty(1+o(1))\frac{\beta-1}{\varrho^{\alpha-1}}H_\tau(\xi)\,\drm\xi
		\right\}\,\drm s
		\end{equation}
		as $t\to\infty$, where 	$M=\lim_{t\to\infty}y^{[1]}(t)\in(0,\infty)$, for some $A\in\R$,
		and $\mS_{\NRV}(\varrho)=\mS_{\RV}(\varrho)=\IS=\IS_{\infty,B}$.
		Moreover, $M-y^{[1]}\in\SV$ and
		\begin{equation} \label{LLM}
		 \frac{t^{\delta+\alpha}(\tau'(t))^{\alpha-1}L_p(t)}{\tau^{\delta+\alpha}(t)L_r(t)(M-y^{[1]}(t))}=o(1)
		\end{equation}
		as $t\to\infty$.
	\end{theorem}

	\begin{remark}
		(i)	In fact -- as a closer examination of the proof shows -- because of
		the first additional condition in \eqref{tau-add}, we could assume only $p\in\RV(\delta)$
		instead of \eqref{pr}.
		Yet, for the proof of regular variation of solutions and asymptotic formulae, we could drop the requirement of $p$ being regularly varying (similarly as $r$ does not need to be regularly varying in Theorem~\ref{T:SV}).
		But, for simplicity and since we want to
		express some formulae in terms of $L_p, L_r$, and thus being in a better correspondence with Theorem~\ref{T:SV},  we prefer to take
		the stronger assumption.
		
		(ii) In view of Proposition~\ref{P:RV}, the former condition in
		\eqref{tau-add} is implied by $\tau'\in\SV$ (provided we assume
		$r\in\RV(\delta+\alpha)$).  The latter condition in
		\eqref{tau-add} is implied by \eqref{Lpr}, and
		$\limsup_{t\to\infty}\tau'(t)<\infty$ (provided we assume
		\eqref{pr}).
		
		(iii)
		Similarly as in Lemma~\ref{L:nec}, we can obtain a
		necessary condition in the setting of Theorem~\ref{T:RV}.
		In particular, if $\delta<-1$, $\tau(t)\asymp t$ as
		$t\to\infty$, \eqref{pr} is satisfied, and $\mathcal{S}_\NRV(\varrho)\ne\emptyset$,
		then \eqref{Lpr} holds.
	\end{remark}

	 In \cite{manojl-tanig,tanigawa2012} (see also \cite{tanigawa2008}), half-linear differential equations with both retarded and advanced arguments are considered.
		Necessary and sufficient conditions for
		the existence of a (generalized) $\SV$ solution and a (generalized) $\RV(1)$ solution are established.
		Note that the conditions on the coefficients are more general than ours and takes an integral
		form similar to \eqref{trintp} and \eqref{trp1}.
		The conditions on the delayed argument are slightly more special than our ones. The methods in \cite{manojl-tanig,tanigawa2012} are based on the results
		for the associated equation without deviating arguments
		and the Schauder-Tychonoff fixed point theorem.
		We emphasize that while in \cite{manojl-tanig,tanigawa2008,tanigawa2012} the existence of $\RV$ solutions is studied, here we attempt to deal with all positive increasing solutions.

		Our results are essentially new also in the linear case, i.e., when $\alpha=2$. The linear version of the existence results mentioned in the previous item can be found in
		\cite{kusano-maric2006,kusano-maric2007georg,kusano-maric2007bull}.

		If $\tau(t)=t$, Theorems \ref{T:SV} and \ref{T:RV} reduce to an improvement of the results in \cite{pr-methods,rt}; note that in those papers, conditions \eqref{pr} and \eqref{Lpr} are assumed throughout.
		%	As a special case of Theorem~\ref{T:SV1} (when $\alpha=2$,
		%	$r(t)=1$, $\tau(t)=t$ and considering just formulae \eqref{f1}, \eqref{f2}) we get
		%	\cite[Theorem~0.1-A]{geluk1990}.
	
		\medskip
	
	Next we establish a generalization of the previous theorems. As we
	will see, as a by-product, we solve the natural problem arising
	from Theorem \ref{T:SV} and \ref{T:RV}, namely the
	missing (``critical'')
	case $\delta=-1$. Also some equations that do not have
	$\RV$ coefficients can be treated by generalized theorems.
	
	We distinguish the two cases, namely
	\begin{equation} \label{rdiv}
	\int_a^\infty r^{1-\beta}(s)\,\drm s=\infty
	\end{equation}
	and
	\begin{equation} \label{rconv}
	\int_a^\infty r^{1-\beta}(s)\,\drm s<\infty.
	\end{equation}
	In the former case we denote $R_D(t)=\int_a^t r^{1-\beta}(s)\,\drm
	s$ and in the latter case we denote $R_C(t)=\int_t^\infty
	r^{1-\beta}(s)\,\drm s$. Further, $R_D^{-1}$ stands for the inverse
	of $R_D$ and $Q^{-1}$ is the inverse of $Q$, where $Q=1/R_C$.
	It is easy to see that $(R_C\circ Q^{-1})(s)=1/s$.
	
	We first give a generalization of Theorem~\ref{T:RV}. Denote
	\begin{equation} \label{pD}
	\tau_D =R_D\circ\tau\circ R_D^{-1}, \ \ p_D=(pr^{\beta-1})\circ
	R_D^{-1},
	\end{equation}
	and
	$$
	%q_{D1}=(R_Dpr^{\beta-2})^{\beta-1}, \ \
	q_{D}=((\tau'_D\circ
	R_D)R_D)^{\alpha-1}p.
	$$
	The following set of conditions will play a role:
\begin{equation}\label{tau2}
\limsup_{t\to\infty}\frac{t}{\tau_i(t)}<\infty,\ \tau_i'\in\SV,\
\limsup_{t\to\infty}\tau'_i(t)<\infty,
\end{equation}
where $i=D$ or $i=C$ according to whether $\delta<-1$ or
$\delta>-1$, respectively.

	\begin{theorem} \label{T:gen1}
		Assume that \eqref{rdiv} holds, $p_D\in\RV(-\alpha)$, and
		$\lim_{t\to\infty}L_{p_D}(t)=0$. Let \eqref{tau} and \eqref{tau2}
		with $i=D$ be fulfilled.
		Then
		for any $y\in\IS$ one has $y\circ R^{-1}_D\in\NRV(1)$ and the following hold.
		
		(i) If $\int_a^\infty q_{D}(s)\,\drm s=\infty$, then
		\begin{equation} \label{tF11}
		y(t)=R_D(t)\exp\left\{\int_a^t(1+o(1))(\beta-1)q_{D2}(s)\,\drm
		s\right\}
		\end{equation}
		as $t\to\infty$ and $y\in\IS_{\infty,\infty}$.
		
		(ii) If $\int_a^\infty q_{D}(s)\,\drm s<\infty$, then
		\begin{equation} \label{tF22}
		y(t)=A+\int_a^t
		M^{\beta-1}r^{1-\beta}(s)\exp\left\{-\int_s^\infty(1+o(1))(\beta-1)q_{D}(u)\,\drm
		u\right\}\,\drm s
		\end{equation}
		as $t\to\infty$, and $y\in\IS_{\infty,B}$, with
		$\lim_{t\to\infty}y^{[1]}(t)=M\in(0,\infty)$, for some $A\in\R$.
		In addition, $|y^{[1]}\circ R_D^{-1}-N|\in\SV$
		and
		\begin{equation} \label{trNo2}
		\frac{(\tau'_D(R_D(t)))^{\alpha-1}R_D^\alpha(t)p(t)r^{\beta-1}(t)}{M-y^{[1]}(t)}=o(1)
		\end{equation}
		as $t\to\infty$.
		
	\end{theorem}
	
	The next result is a complement of Theorem~\ref{T:gen1}, and
	generalizes Theorem~\ref{T:SV}. Denote
	\begin{equation} \label{pC}
	\tau_C =Q\circ\tau\circ Q^{-1}, \ \ p_C=(R_C^2 pr^{\beta-1})\circ
	Q^{-1}
	\end{equation}
	and
	$$
	q_{C}=(R_Cpr^{\beta-1})^{\beta-1}.
	%, \ \ q_{C2}=\frac{(\tau'_C\circ
	%Q)^{\alpha-1}p}{(\tau_C\circ Q)^{2\alpha-2}R_C^{\alpha-1}}.
	$$

	\begin{theorem} \label{T:gen2}
		Assume that \eqref{rconv} holds, $p_C\!\in\!\RV(\alpha-2)$, and
		$\lim_{t\to\infty}L_{p_C}(t)=0$. Let \eqref{tau} and \eqref{tau2}
		with $i=C$ be fulfilled.
		Then for any $y\in\IS$ one has $y\circ Q^{-1}\in\NSV$ and the
		following hold.
		
		(i) If $\int_a^\infty q_{C}(s)\,\drm s=\infty$, then
		\begin{equation} \label{tF1c}
		y(t)=\exp\left\{\int_a^t(1+o(1))(\beta-1)^{\beta-1}q_{C}(s)\,\drm
		s\right\}
		\end{equation}
		as $t\to\infty$ and $y\in\IS_{\infty,\infty}$.
		
		(ii) If $\int_a^\infty
		q_{C}(s)\,\drm s<\infty$, then
		\begin{equation} \label{tF2c}
		y(t)=N\exp\left\{-\int_t^\infty(1+o(1))(\beta-1)^{\beta-1}q_{C}(s)\,\drm
		s\right\}
		\end{equation}
		as $t\to\infty$ and $y\in\IS_{B,\infty}$, with
		$\lim_{t\to\infty}y(t)=N\in(0,\infty)$. In addition, $|y\circ Q^{-1}-N|\in\SV$ and
		\begin{equation} \label{trNo1c}
		\frac{R^\alpha_C(t)p(t)r^{\beta-1}(t)}{\Phi(y(t)-N)}=o(1)
		\end{equation}
		as $t\to\infty$.
	\end{theorem}
	
	\begin{remark} \label{R:tau}
		In view of Lemma~\ref{L:tau}-(ii), if $\tau$ satisfies \eqref{tau}, \eqref{tau+}, $\tau'\in\SV$,
		$\limsup_{t\to\infty}\tau'(t)<\infty$, and
		$r\in\RV(\delta+\alpha)$, then \eqref{tau2} (which is assumed in Theorems~\ref{T:gen1} and \ref{T:gen2}) is fulfilled,
		where $i=D$ or $i=C$ according to whether $\delta<-1$ or
		$\delta>-1$, respectively.
	\end{remark}

		Theorems \ref{T:gen1} and \ref{T:gen2} are indeed generalizations
		of the previous theorems. As easily seen from Lemma~\ref{L:tau} (see also Remark~\ref{R:tau}), the assumptions of Theorem~\ref{T:RV}
		and Theorem~\ref{T:SV} (supposing here \eqref{pr})
		imply the ones of Theorem~\ref{T:gen1} and Theorem~\ref{T:gen2}, respectively. Asymptotic formulae in
		Theorems \ref{T:SV} and \ref{T:RV} can be obtained from the general ones by applying the Karamata integration theorem (Proposition~\ref{T:karam}) to $R_D$ and $R_C$.

\medskip

 Theorem~\ref{T:gen1} and Theorem~\ref{T:gen2} allow us to obtain asymptotic formulae
		also when the coefficients of the equation  are not regularly varying.
		For example, let $p(t)=e^{\gamma t}t^\omega$, $r(t)=e^{\gamma t}$
		with $\gamma<0$ and $\omega<0$. Then \eqref{rdiv} holds and
		$$
		 p_D(s)=\frac{1}{((s+K)\gamma(1-\beta))^\alpha}\ln^\omega\left((s+K)\gamma(1-\beta)\right)^{\frac{1}{\gamma(1-\beta)}}
		$$
		for some $K\in(0,\infty)$. Therefore, $p_D\in\RV(-\alpha)$ and
		$L_{p_D}(s)\to 0$ as $s\to\infty$.

\medskip

 Theorem~\ref{T:gen1} and Theorem~\ref{T:gen2}  enable us to
		cover the missing cases in Theorems \ref{T:SV} and \ref{T:RV} in the sense of the critical case $\delta=-1$. Let us illustrate this
		fact if, for simplicity, $r(t)=t^{\alpha-1}$.
		Assume, in addition, that $p\in\RV(-1)$. Then \eqref{pr} with
		$\delta=-1$ holds. It is easy to see that \eqref{rdiv} is
		fulfilled. We have $R_D(t)=\ln t$ (taking $a=1$) and
		$R_D^{-1}(s)=\exp s$. Since
		$p_D(s)=p(e^s)(e^s)^{(\alpha-1)(\beta-1)}=e^sp(e^s)=L_p(e^s)$, we
		get $p_D\in\RV(-\alpha)$ provided
		\begin{equation} \label{to1}
		L_p\circ\exp\in\RV(-\alpha).
		\end{equation}
		Under this condition, $L_{p_D}(\ln t)=s^\alpha p_D(s)=s^\alpha
		L_p(e^s)=L_p(t)\ln^\alpha t$ (with $t=e^s$). Hence,
		$\lim_{t\to\infty}L_{p_D}(t)=0$ is implied by
		\begin{equation} \label{to2}
		\lim_{t\to\infty}L_p(t)\ln^\alpha t =0.
		\end{equation}
		Thus the assumptions of Theorem~\ref{T:gen1} are fulfilled
		provided \eqref{to1} and \eqref{to2} hold, and we get
		$y\circ\exp\in\NRV(1)$ for
		$y\in\IS$; in fact, all eventually positive increasing solutions of \eqref{e} are slowly
		varying. As for the asymptotic formulae,
		note that
		$$
		q_{D}(t)=\left(\frac{t\tau'(t)}{\tau(t)}\right)^{\alpha-1}p(t)\ln^{\alpha-1}s.
		$$
		Take, for example,
		$$
		p(t)=\frac{1}{t}\cdot\frac{1}{\ln^\alpha t}\cdot\frac{1}{\ln^\omega\ln t},
		$$
		$\omega>0$.
		Then
		$$
		(L_p\circ\exp)(t)=\frac{1}{t^\alpha\ln^{\omega}t}\in\RV(-\alpha).
		$$
		Assume $s\tau'(s)\asymp \tau(s)$ as $s\to\infty$.
		Then
		$$
		q_{D}(t)=\left(\frac{t\tau'(t)}{\tau(t)}\right)^{\alpha-1}\frac{\ln^{-\omega}\ln t}{t\ln t}
		\asymp\frac{\ln^{-\omega}\ln t}{t\ln t}
		$$
		as $t\to\infty$.
		Hence,
		$\int_a^\infty q_{D}$ diverges [converges] when $\omega\le 1$ [$\omega>1$].
%\label{R:open}
		
\medskip

Many formulae in the previous two theorems (and their applications)
		can be expressed in terms of generalized regularly varying
		functions; this concept was introduced in \cite{jaros-kusano} and has been used in several papers, see, for instance, \cite{tanigawa2012}.
		Generalized regular variation is defined as follows: a function $f$ is
	\textit{regularly varying of index $\vartheta\in\R$ with respect to} $\omega\in C^1$, with $\omega'>0$ and $\lim_{t \to \infty}\omega(t)=\infty$, if $f\circ \omega^{-1}\in\RV(\vartheta)$. Denote the set of all regularly varying function of index $\vartheta$ with respect to $\omega$ by $\RV_\omega(\vartheta)$.
	Hence, for example, instead of $y\circ R_D^{-1}\in\NRV(1)$ we could
		write $y\in\NRV_{R_D}(1)$, instead of $p_D\in\RV(-\alpha)$ we could write
		$pr^{1-\beta}\in\RV_{R_D}(-\alpha)$, and so on.

\medskip
	
		If $\tau(t)=t$, then Theorem~\ref{T:gen1}
		(Theorem~\ref{T:gen2}) reduces to \cite[Theorem~5]{pr-asymp}
		(\cite[Theorem~6]{pr-asymp}).

\medskip
In the last paragraph of this section
 we indicate some directions for a possible
		future research. 	
		
		(i) One of the open problems is to obtain a similar theory for
		equation \eqref{e} with the opposite sign condition on the
		coefficient $p$. In this framework, it would be new even in the
		ordinary (i.e., non-functional) case.
		
		(ii)
		A natural problem is to complete our theory also for decreasing solutions.
		As we could see in the observation after formula \eqref{trp1}, the difference between functional and non-functional case is ``bigger'' than for increasing solutions. We conjecture that, for example, $\int_a^\infty r^{1-\beta}(s)\,\drm s=\infty$, $p\in\RV(\delta)$ with $\delta<-1$, and \eqref{trp1} imply $\DS\subset\NSV$; the assumption on regular variation of $p$ (or some other restriction on $p$) cannot be omitted. Once we have guaranteed slow variation of decreasing solutions, asymptotic formulae similar to the above ones can be obtained.
		
		(iii)
		Another natural problem is to examine advanced equations, i.e., the case in which the condition $\tau(t)\le t$ is replaced by $\tau(t)\ge t$. Here the situation is ``reversed'' in the sense that  decreasing solutions are easier to be handled than  increasing ones. For instance, it is not difficult to show that  $\int_a^\infty r^{1-\beta}(s)\,\drm s=\infty$, $p\in\RV(\delta)$ with $\delta<-1$, and \eqref{trp1} (i.e., the example of the setting is the same as in the previous item) imply $\DS\subset\NSV$. Such a statement would be a half-linear extension of  \cite[Theorem~5.1]{grimm-hall} which deals with linear advanced equations. To the best of authors' knowledge, this paper is the first one where functional differential equations are analyzed in the framework of regular variation.
		As for the corresponding result for an advanced equation when \eqref{trintp} with $\delta>-1$ holds, nothing is known about slow variation of all increasing solutions
		even in the linear case.
		On the other hand, practically all arguments which are used to obtain asymptotic formulae, work with appropriate modifications also for advanced equations.
		
		(iv)
		Since we assume that ${L_p(t)}/{L_r(t)}={t^\alpha p(t)}/{r(t)}$  tends to zero as $t \to \infty$,
		(i.e., condition \eqref{Lpr}), it is natural to examine also the
		condition $\lim_{t\to\infty}{t^\alpha p(t)}/{r(t)}$ $=C>0$ or its generalization in the sense of \eqref{trintp} or \eqref{trp1}.
		In the case of  equations without deviating argument, see \cite{pr-asymp},
		we can use suitable transformations which lead to a linear second order equation.
%into a generalized Riccati equation,
%		and then another transformation into a modified Riccati equation which leads to a linear second order equation.
%This procedure leads to asymptotic linearization which somehow substitutes a certain linear transformation
%		of dependent variable that is not at disposal in the half-linear
%		case.
Then the results on $\SV$ solutions can be applied to the transformed equation.
		A similar method however is not known for the associated
		functional differential equations (even in the linear case).

	\section{Proofs of the main results}
	
In order to prove the main theorems, first we derive several auxiliary statements.	
	
	\begin{proof}[Proof of Proposition~\ref{P:RV}-(ix)]
		(ix) If $f'\in\RV(-1)$ and $f(t)\to \infty$ as $t\to\infty$, then $f(t)=f(a)+\int_a^t f'(s)\,\drm s\sim\int_a^t f'(s)\,\drm s$. Hence, since $t f'(t) \in \SV$,
		$$tf'(t)/f(t)\sim tf'(t)/\int_a^t f'(s)\,\drm s%\sim\int_a^t f'(s)\,\drm s
\to 0$$ as $t\to\infty$ by Proposition~\ref{T:karam}-(iii). Similarly we proceed when $|f'|\in\RV(-1)$ and $f(t)\to 0$ as $t\to\infty$. As for the case $\vartheta\ne-1$, we use again similar arguments and apply Proposition~\ref{T:karam}-(i) when
		$\vartheta<-1$, or Proposition~\ref{T:karam}-(ii)  when $\vartheta>-1$.
	\end{proof}

	\begin{lemma} \label{L:fg} Let $f,g$ be defined on $[a,\infty)$.
		
		(i) If $f\in\SV$ and $g(t)\asymp t$ as $t\to\infty$, then
		$f(g(t))\sim f(t)$ and $f\circ g\in\SV$.
		
		(ii) If $f\in\RV(\vartheta)$, $\vartheta\in\R$, and $g(t)\asymp t$
		as $t\to\infty$, then $f(g(t))\asymp f(t)$.
		
		(iii) If $f\in\RV(\vartheta)$, $\vartheta\in\R$, and $g(t)\sim t$
		as $t\to\infty$, then $f(g(t))\sim f(t)$.
	\end{lemma}
	
	\begin{proof}
		(i) Since $c\le g(t)/t\le d$, $t\in[b,\infty)$, for some
		$0<c<d<\infty$ and $b \geq a$, in view of the Uniform Convergence
		Theorem, we have
		$$
		\frac{f(g(t))}{f(t)}=\frac{f((g(t)/t)t)}{f(t)}\to 1
		$$
		as $t\to\infty$, i.e., $f(g(t))\sim f(t)$. Therefore, for every $\lambda>0$,
		$$
		\frac{f(g(\lambda t))}{f(g(t))}\sim\frac{f(\lambda t)}{f(t)}\sim 1
		$$
		as $t\to\infty$, i.e., $f \circ g\in\SV$.
		
		(ii) Let $L(t)=f(t)/t^\vartheta$. Then $L\in\SV$ and
		\begin{equation} \label{fgf}
		\frac{f(g(t))}{f(t)}=\left(\frac{g(t)}{t}\right)^\vartheta\frac{L(g(t))}{L(t)},
		\end{equation}
		In view of (i), we have
		$$
		c_1\le c_2\frac{L(g(t))}{L(t)}\le\frac{f(g(t))}{f(t)}\le
		d_2\frac{L(g(t))}{L(t)}\le d_1,
		$$
		$t\in[b,\infty)$, for some $c_1,c_2,d_1,d_2\in(0,\infty)$ and $b \geq a$.
		
		(iii) As above,  let $L(t)=f(t)/t^\vartheta$. Then $L\in\SV$ and,
		in view of \eqref{fgf} and (i), we have
		$$
		\frac{f(g(t))}{f(t)}\sim\frac{L(g(t))}{L(t)}\sim 1
		$$
		as $t\to\infty$.
	\end{proof}
	
	\begin{lemma} \label{L:tau}
		Let $\tau_i,p_i$, $i=C,D$, are defined as in \eqref{pD} and
		\eqref{pC}.
		
		(i) If \eqref{pr} and \eqref{Lpr} hold, then $p_D\in\RV(-\alpha)$
		and $L_{p_D}(s)\to 0$ as $s\to\infty$ provided $\delta<-1$, while
		$p_C\in\RV(\alpha-2)$ and $L_{p_C}(s)\to 0$ as $s\to\infty$
		provided $\delta>-1$.
		
		(ii) Let $\tau$ satisfy \eqref{tau}. Then $\tau_i$ fulfills
		\begin{equation}\label{tau1}
		\tau_i\in C^1,\ \tau'_i>0,\ \tau_i(t)\le t,
		\end{equation}
		where $i=D$ or $i=C$ according to whether \eqref{rdiv} or
		\eqref{rconv} holds, respectively. If, in addition, \eqref{tau+} holds, $\tau'\in\SV$,
		$\limsup_{t\to\infty}\tau'(t)<\infty$, and
		$r\in\RV(\delta+\alpha)$, then \eqref{tau2} holds,
		where $i=D$ or $i=C$ according to whether $\delta<-1$ or
		$\delta>-1$, respectively.
	\end{lemma}
	
	\begin{proof}
		(i) The statement follows from Proposition~\ref{P:RV}, the
		Karamata integration theorem (Proposition~\ref{T:karam}), and the fact
		that $L_p(t)/L_r(t)\sim\varrho^\alpha(L_{p_D}\circ R_D)(t)$ as
		$t\to\infty$ provided $\delta<-1$ (the case $\delta>-1$ is
		similar). The details are left to the reader.
		
		(ii) We will consider only the case $i=D$, since the case $i=C$ is
		similar. Since $\tau, R_D, R_D^{-1}$ are $C^1$ functions with
		positive derivatives, $\tau_D\in C^1$ and $\tau_D'>0$ follow.
		Further, from $\tau(t)\le t$, we have $\tau(R_D^{-1}(s))\le
		R^{-1}_D(s)$, and since $R_D$ is increasing, we obtain $\tau_D(s)\le
		s$. From now on assume that $r\in\RV(\delta+\alpha)$. Then
		$R_D\in\RV(\varrho)$ by Proposition~\ref{P:RV} and the Karamata
		theorem (Proposition~\ref{T:karam}). If  $M>0$ exists, such that $t\le
		M\tau(t)$, then $R^{-1}_D(s)\le M\tau(R_D^{-1}(s))$. Consequently,
		$s\le R_D(M\tau(R^{-1}_D(s)))\le N\tau_D(s)$, for some $N>0$,
		where the existence of $N$ is guaranteed by regular variation of
		$R_D$. Hence, $\limsup_{t\to\infty}t/\tau_D(t)<\infty$ follows.
		Since $\tau'\in\SV$ and $R_D\in\RV(\varrho)$, we have
		$\tau\in\RV(1)$ and $R_D^{-1}\in\RV(1/\varrho)$. Thus, in view of
		Proposition~\ref{P:RV}, $$
		\begin{aligned}
		\tau'_D&=(R'_D\circ\tau\circ
		R^{-1}_D)(\tau'\circ R^{-1})(R^{-1})'\\
		&\in\RV\left((\varrho-1)\cdot
		1\cdot
		\frac{1}{\varrho}+0\cdot\frac{1}{\varrho}+\left(\frac{1}{\varrho}-1\right)\right)=\SV.
		\end{aligned}
		$$
		Finally, assume that $\tau'(t)\le M$ for some $M>0$. In view of
		Lemma~\ref{L:fg}, we have
		$r^{1-\beta}(\tau(t))/r^{1-\beta}(t)\asymp 1$ as $t\to\infty$.
		Hence,
		$$
		\tau'_D(s)=r^{1-\beta}(\tau(R^{-1}_D(s)))\tau'(R^{-1}_D(s))\frac{1}{r^{1-\beta}(R^{-1}_D(s))}\le
		N
		$$
		for some $N>0$.

	\end{proof}

	\begin{lemma} \label{L:rec}
		Let $y$ be a solution of \eqref{e} and let $\tau$ be differentiable. Then $u=Cr\Phi(y')$, $C\in\R$, satisfies the reciprocal equation
		\begin{equation} \label{re}
		(\widetilde r(t)\Phi^{-1}(u'))'=\widetilde p(t)\Phi^{-1}(u(\tau(t))),
		\end{equation}
		where
		\begin{equation} \label{re-coef}
		\widetilde r(t)=p^{1-\beta}(t),\quad  \widetilde
		p(t)=\tau'(t)r^{1-\beta}(\tau(t)).
		\end{equation}
		If $p\in\RV(\delta)$ and
		$(r^{1-\beta}\tau'\circ\tau)\in\RV(\widetilde\delta)$ (where the
		latter condition is implied e.g. by $r\in\RV(\delta+\alpha)$ and
		$\tau'\in\SV$), then $\widetilde p\in\RV(\widetilde\delta)$,
		$\widetilde r\in\RV(\widetilde\delta+\beta)$, where
		$\widetilde\delta=\delta(1-\beta)-\beta$. If, moreover, \eqref{pr}
		holds, then
		\begin{equation} \label{tildeLL}
		\frac{L_{\widetilde p}(t)}{L_{\widetilde
				r}(t)}\asymp\left(\frac{L_p(t)}{L_r(t)}\right)^{\beta-1}\tau'(t)
		\end{equation}
		as $t\to\infty$.
		%If \eqref{Lpr} holds, then
		%\begin{equation} \label{tildeLL2}
		%\frac{L_{\widetilde p}(t)}{L_{\widetilde r}(t)}\to 0
		%\end{equation}
		%as $t\to\infty$.
	\end{lemma}
	
	\begin{proof} Since $y$ is a solution of \eqref{e}, from the definition of $u$ we have $u'(t)=p(t)\Phi(y(\tau))$.
		Hence $y(\tau(t))=\Phi^{-1}(1/p(t))\Phi^{-1}(u(t))$. Further, $y'(t)=\Phi^{-1}(1/t(t))$ $\Phi^{-1}(u(t))$
		and $(y\circ\tau)'=(y'\circ\tau)\tau'$. Now it is
 easy to see that $u$ satisfies \eqref{re}. The fact that $\widetilde
		p\in\RV(\widetilde\delta)$ and $\widetilde
		r\in\RV(\widetilde\delta+\beta)$ follows from
		Proposition~\ref{P:RV}. In view of Lemma~\ref{L:fg}, we have
		$$
		\frac{L_{\widetilde p}(t)}{L_{\widetilde
				r}(t)}=\frac{t^\beta\tau'(t)r^{1-\beta}(\tau(t))}{p^{1-\beta}(t)}\asymp\left(
		\frac{t^\alpha p(t)}{r(t)}\right)^{\beta-1}\tau'(t)
		$$
		as $t\to\infty$, which implies \eqref{tildeLL}.
		%Finally,
		%\eqref{tildeLL2} follows from \eqref{tildeLL} and \label{tau}.
	\end{proof}

	\begin{lemma} \label{L:transf}
		Put $s=\varphi(t)$ and $x(s)=y(\varphi^{-1}(s))$, where $\varphi$ is a differentiable function such that $\varphi'(t)\ne 0$,
		equation \eqref{e} is transformed into the equation
		\begin{equation} \label{te}
		\frac{\drm}{\drm s}\left(\widehat r(s)\Phi\left(\frac{\drm x}{\drm s}\right)\right)=\widehat p(s)\Phi(x(\widehat \tau(s))),
		\end{equation}
		where
		$$
		\widehat p=\frac{p\circ\varphi^{-1}}{\varphi'\circ\varphi^{-1}},\ \
		\widehat r=(r\circ\varphi^{-1})\Phi(\varphi'\circ\varphi^{-1}),\ \
		\widehat\tau =\varphi\circ\tau\circ\varphi^{-1}.
		$$
		Further,
		\begin{equation} \label{x1y1}
		x^{[1]}(s):=\widehat r(s)\Phi\left(\frac{\drm x}{\drm s}(s)\right)=y^{[1]}(t).
		\end{equation}
		If \eqref{pr}, \eqref{rdiv}, and $\varphi=R_D$ hold, then $\widehat r=1$ and $\widehat p=(pr^{\beta-1})\circ R_D^{-1}\in\RV(-\alpha)$.
		If \eqref{pr}, \eqref{rconv}, and $\varphi=Q$ hold, then $\widehat r(s)=s^{2\alpha-2}$ and $\widehat p=(R_C^2pr^{\beta-1})\circ Q^{-1}\in\RV(\alpha-2)$.
	\end{lemma}
	
	\begin{proof}
		The form of the transformed equation follows from the fact that
		$x(\varphi(t))$$=y(t)$ and $\frac{\drm}{\drm
			t}=\varphi'(t)\frac{\drm}{\drm s}$. The indices of regular
		variation of $\widetilde p$ in both cases can easily be computed
		via Proposition~\ref{P:RV} and the Karamata theorem (Proposition~\ref{T:karam}).
	\end{proof}
	
	%\begin{lemma}
	%Assume that \eqref{pr} holds. If $\delta<-1$, then $\mS\cap\SV\subseteq\DS$.
	%If $\delta>-1$, then $\mS\cap\SV\subseteq\IS$.
	%\end{lemma}
	
	%\begin{proof}
	%\textcolor{blue}{I will add the proof later.}
	%\end{proof}

	\begin{proof}[Proof of Theorem~\ref{T:SV}]
		\textit{{(I)}} Take $y\in\IS$.
		Integrating \eqref{e}
		from $b$ to $t$, where $b$ is so large that $y(\tau(s))>0$ for
		$s\ge b$, we get
		\begin{equation} \label{intetb}
		\begin{aligned}
		r(t)\Phi(y'(t)) &= r(b)\Phi(y'(b))+\int_b^t
		p(s)\Phi(y(\tau(s)))\,\drm s\\
		&\ge r(b)\Phi(y'(b))+\Phi(y(\tau(b)))\int_b^t p(s)\,\drm
		s\to\infty
		\end{aligned}
		\end{equation}
		as $t\to\infty$, since $\int_a^\infty p(s)\,\drm s=\infty$. Thus $\lim_{t\to\infty}y^{[1]}(t)=\infty$. The
		equality in \eqref{intetb} and the divergence of the integral lead
		to the existence of $B>0$ such that
		$$
		\begin{aligned}
		r(t)\Phi(y'(t)) &\le B\int_b^t p(s)\Phi(y(\tau(s)))\,\drm s \le
		B\Phi(y(\tau(t)))\int_b^t p(s)\,\drm s\\ &\le B\Phi(y(t))\int_b^t
		p(s)\,\drm s
		\end{aligned}
		$$
		for large $t$. Hence,
		$$
		0<\left(\frac{ty'(t)}{y(t)}\right)^{\alpha-1}
		\le\frac{Bt^{\alpha-1}}{r(t)}\int_b^t p(s)\,\drm s,
		$$
		where the right-hand side tends to zero as $t\to\infty$ by our
		assumptions. Therefore, $ty'(t)/y(t)\to 0$ as
		$t\to\infty$, and so $y\in\NSV$.
		Note that conditions \eqref{tau}, \eqref{tau+} imply $\tau \asymp t$.  Hence, in view of
		Lemma~\ref{L:fg}, we have $y(t)\sim y(\tau(t))$ as $t\to\infty$. Moreover, if $p \in \RV(\delta)$ with $\delta>-1$, then  $p\Phi(y)\in\RV(\delta)$ by Proposition~\ref{P:RV}.
		Thus, recalling the equality in \eqref{intetb},  Proposition~\ref{T:karam} yields
		$$
		\begin{aligned}
		r(t)\Phi(y'(t)) &\sim\int_b^t p(s)\Phi(y(\tau(s)))\,\drm s\sim
		\int_b^t p(s)\Phi(y(s))\,\drm s\\
		&=\int_b^t s^\delta
		L_p(s)y^{\alpha-1}(s)\,\drm s\\
		&\sim\frac{1}{\delta+1}t^{\delta+1}L_p(t)y^{\alpha-1}(t)
		=\frac{1}{\delta+1}tp(t)y^{\alpha-1}(t)
		\end{aligned}
		$$
		as $t\to\infty$. This implies that $y^{[1]} \in \RV(\delta+1)$, $\delta+1>0$, therefore $\lim_{t \to \infty} y^{[1]}(t)=\infty$. Further, from the above asymptotic relation we have
		\begin{equation} \label{y'y}
		\frac{y'(t)}{y(t)}=(1+o(1))\Phi^{-1}\left(\frac{tp(t)}{(\delta+1)r(t)}\right)
		\end{equation}
		as $t\to\infty$.
		
If $\int_a^\infty G(s)\,\drm s=\infty$, then we
		integrate \eqref{y'y} from $b$ to $t$, to get
		\begin{equation} \label{lny}
		\ln y(t)=\ln
		y(b)+\int_b^t(1+o(1))\frac{G(s)}{\Phi^{-1}(\delta+1)}\,\drm s
		=\int_a^t(1+o(1))\frac{G(s)}{\Phi^{-1}(\delta+1)}\,\drm s,
		\end{equation}
		where the latter equality is true thanks to the divergence of
		$\int_a^\infty G(s)\,\drm s$. Indeed, it easily follows that for $\varepsilon(t)\to 0$ there are $\varepsilon_1(t), \varepsilon_2(t)\to0$ (as $t\to\infty$)
		such that
$$
c+\int_b^t(1+\varepsilon(s))G(s)\,\drm s=(1+\varepsilon_1(t))\int_a^t(1+\varepsilon(s))G(s)\,\drm s
=\int_a^t(1+\varepsilon_2(s))G(s)\,\drm s,
$$	where $c$ is constant.
		 Relation \eqref{lny} readily implies
		formula \eqref{f1}. Moreover, we have $\lim_{t\to\infty}y(t)=\infty$,
		and so $\IS\subseteq\IS_{\infty,\infty}$.

If $\int_a^\infty G(s)\,\drm
		s<\infty$, then we integrate \eqref{y'y} from $t$ to $\infty$, to
		get
		$$ %\begin{equation} \label{lny}
		\ln N-\ln y(t)=\int_t^\infty(1+o(1))\frac{G(s)}{\Phi^{-1}(\delta+1)}\,\drm s,
		$$ %\end{equation}
		where $N=\lim_{t\to\infty}y(t)$, and formula \eqref{f2} follows.
		Clearly, $N$ has to be finite because of convergence of the
		integral, and so $\IS\subseteq\IS_{B,\infty}$.
		
Now we show
		that $\SV$ solutions cannot decrease (when $\delta>-1$), i.e., $\mS_{\SV}\subseteq\IS$. Take
		$y\in\DS$. Since $y^{[1]}$ is negative increasing, there exists
		$\lim_{t\to\infty}y^{[1]}(t)=K\in(-\infty,0]$. Suppose now that
		$y\in\SV$. Then, in view of Lemma~\ref{L:fg} and Proposition~\ref{P:RV},
		$py^{\alpha-1}\circ\tau\in\RV(\delta)$. Hence, $\int_b^t
		p(s)\Phi(y(\tau(s)))\,\drm s\to\infty$ as $t\to\infty$ since
		$\delta>-1$. This is however a contradiction, which can be seen from the equality in \eqref{intetb}. Since $\mathcal{S}=\IS\cup\DS$, we have proved that
		$\mS_{\SV}\subseteq\IS$. The other relations between the classes
		$\IS,\mS_{\SV},\mS_{\NSV}$, and $\IS_{x,\infty}$ with $x=\infty$
		or $x=B$ are clear.
		
\medskip \textit{(II)}		
If $p\in\RV(\delta)$ and $r\in\RV(\gamma)$, then, because of Proposition~\ref{T:karam},  condition \eqref{trintp} reads as
		$$
		\lim_{t\to\infty}\frac{t^{\alpha+\delta-\gamma}L_p(t)}{L_r(t)}=0,
		$$
		from which we necessarily obtain $\gamma\ge\alpha+\delta$ by Proposition~\ref{P:RV}-(vii).
		Let $\gamma=\alpha+\delta$ and $y\in\IS$. Then, in view of Lemma~\ref{L:fg} and Proposition~\ref{P:RV},
		$y\circ\tau\in\SV$, thus $(r\Phi(y'))'=p\Phi(y\circ\tau)\in\RV(\delta)$. Since $r(t)\Phi(y'(t))\sim\int_t^\infty p(s)y\Phi(\tau(s))\,\drm s$ as $t\to\infty$, we have $r\Phi(y')\in\RV(\delta+1)$ by Proposition~\ref{T:karam}, and so $\Phi(y')\in\RV(\delta+1-\delta-\alpha)=\RV(1-\alpha)$. Thus $y'\in\RV(-1)$ by Proposition~\ref{P:RV} and $y\in\Pi(ty'(t))$ by Proposition~\ref{P:Pi}.
Notice that condition $\gamma>\alpha+\delta$ implies the convergence of $\int_a^\infty G(s)\,\drm s$. Thus, if $\gamma>\alpha+\delta$ or $\gamma=\alpha+\delta$ and $\int_a^\infty G(s)\,\drm s<\infty$,  as we already know, $\lim_{t \to \infty}y(t)=N\in(0,\infty)$, and from \eqref{y'y},
		$$
		y'(t)\sim\frac{1}{\Phi^{-1}(\delta+1)}G(t)y(t)
		\sim\frac{N}{\Phi^{-1}(\delta+1)}G(t)\in\RV((\delta+1-\gamma)(\beta-1))
		$$
		as $t\to\infty$. Integrating this relation from $t$ to $\infty$ and using Proposition~\ref{T:karam}, we obtain \eqref{N-y}.
		If, in addition, $\gamma=\alpha+\delta$, then, in view of $y\in\Pi$, we get $N-y\in\SV$ by Proposition~\ref{P:Pi}; in our case this can easily be seen also from \eqref{N-y}. Formula \eqref{LLN} follows from Proposition~\ref{T:karam}-(iii) since $G(t)=\frac{1}{t}\left(\frac{L_p(t)}{L_r(t)}\right)^{\beta-1}$.
	\end{proof}
	
	\begin{proof}[Proof of Lemma~\ref{L:nec}]
		Take
		$y\in\IS\cap\NSV$. Set $w=r\Phi(y'/y)$. Then $w$ is eventually positive and satisfies
		\begin{equation} \label{ric}
		w'(t)-\frac{\Phi(y(\tau(t)))}{\Phi(y(t))}p(t)+(\alpha-1)r^{1-\beta}(t)w^\beta(t)=0
		\end{equation}
		for large $t$. We have
		$$0<\frac{t^{\alpha-1}}{r(t)}w(t)=\left(\frac{ty'(y)}{y(t)}\right)^{\alpha-1}\to 0$$
		as $t\to\infty$ since $y\in\NSV$.
		Further, $t^{\alpha-1}/r(t)=t^{\alpha-1-\gamma}/L_r(t)\in\RV(\alpha-1-\gamma)$. By our assumptions, $t^{\alpha-1}/r(t)\to 0$ as $t\to\infty$. Denote
		$$
		\Psi(t)=\frac{t^{\alpha-1}}{r(t)}\int_b^t r^{1-\beta}(s)w^\beta(s)\,\drm s,
		$$
		where $b\ge a$ is such that $y(t)>0$ and $y'(t)>0$ for $t\ge b$.
		If $\int_b^\infty r^{1-\beta}(s)w^\beta(s)\,\drm s<\infty$, then clearly $\Psi(t)\to 0$ as $t\to\infty$.
		Let $\int_b^\infty r^{1-\beta}(s)w^\beta(s)\,\drm s=\infty$.
		Without loss of
		generality we may assume $r\in \NRV(\gamma)\cap C^1$.
		Indeed, if $r$ is not normalized or is not in $C^1$, then we can
		take $\tilde r\in \NRV(\gamma)\cap C^1$ with $\tilde
		r(t)\sim r(t)$ as $ t \to \infty $, and we have
		$$ \Psi(t)\sim\frac{t^{\alpha-1}}{\tilde r(t)}\int_b^t \tilde r^{1-\beta}(s)(\tilde r(s)\Phi(y'(s)/y(s)))^\beta \drm s. $$
		The L'Hospital rule yields
		$$
		\begin{aligned}
		\lim_{t\to\infty}\Psi(t)		 &=\lim_{t\to\infty}\frac{r^{1-\beta}(t)w^\beta(t)}{r'(t)t^{1-\alpha}+(1-\alpha)r(t)t^{-\alpha}}\\
		%&=\lim_{t\to\infty}\frac{r^{-\beta}(t)|w(t)|^\beta t^\alpha}{-tr'(t)/r(t)+(\alpha-1)}\\
		 &=\lim_{t\to\infty}\frac{(ty'(t)/y(t))^\alpha}{tr'(t)/r(t)+(1-\alpha)}=\frac{0}{\gamma-\alpha+1}=0.
		\end{aligned}
		$$
		Integrating
		\eqref{ric} from $b$ to $t$ and multiplying by
		$t^{\alpha-1}/r(t)$, we obtain
		\begin{multline} \label{i-ric}
		\frac{t^{\alpha-1}}{r(t)}w(t)-\frac{t^{\alpha-1}}{r(t)}w(b)=\frac{t^{\alpha-1}}{r(t)}\int_b^t
		p(s)\Phi\left(\frac{y(\tau(s))}{y(s)}\right)\,\drm
		s+(\alpha-1)\Psi(t).
		\end{multline}
		In view of the previous observations, from \eqref{i-ric} we obtain
		\begin{equation} \label{anec1}
		\frac{t^{\alpha-1}}{r(t)}\int_b^t
		p(s)\Phi\left(\frac{y(\tau(s))}{y(s)}\right)\,\drm s\to 0
		\end{equation}
		as $t\to\infty$. Since $y\in\SV$ and $\tau(t)\asymp t$ as $t\to\infty$, by Lemma~\ref{L:fg}, there exists $K>0$ such that  $y(\tau(t))/y(t)\ge K$ for $t\ge b$, which, in view of \eqref{anec1}, implies \eqref{trintp}. Assuming in addition $p\in\RV(\delta)$, then \eqref{tLpLr} follows by
		Proposition~\ref{T:karam}.
	\end{proof}
	
	\begin{proof}[Proof of Theorem~\ref{T:RV}]
		%Let $\widetilde \mS,\widetilde \IS,\widetilde \DS,
		%\widetilde\IS_\infty,\widetilde\IS_B,$ have the same meaning with
		%respect to \eqref{re} as $\mS,\IS,\DS,\IS_\infty,\IS_B$,
		%respectively, have with respect to \eqref{e}.
		For the coefficients $\widetilde p,\widetilde r$ defined by
		\eqref{re-coef} we have that $\widetilde
		p\in\RV(\widetilde\delta)$ and $\widetilde
		r\in\RV(\widetilde\delta+\beta)$, where
		$\widetilde\delta:=\delta(1-\beta)-\beta$, and
		$\lim_{t\to\infty}L_{\widetilde p}(t)/L_{\widetilde r}(t)=0$
		thanks to Lemma~\ref{L:rec} and \eqref{tau-add}.
		
		Take $y\in\IS$ and let $u=r\Phi(y')$. Then
		$u\in\IS^{\eqref{re}}$, recalling the notation \eqref{hvezda}. We have $\widetilde\delta>-1$, and thus we
		can apply Theorem~\ref{T:SV} to equation \eqref{re} to obtain
		$u\in\NSV$ and $u\in\Pi(tu')$. Hence,
		$y'\in\RV((-\delta-\alpha)(\beta-1))$
		by Proposition~\ref{P:RV}.
		Since $\delta<-1$, we have $(-\delta-\alpha)(\beta-1)>-1$, and thus
		$\int_b^\infty y'(s)\,\drm s=\infty$. Consequently, $y(t)\sim y(t)-y(b)=\int_b^ty'(s)\,\drm s\in\NRV(\varrho)$ as $t\to\infty$, where $\varrho=(-\delta-1)/(\alpha-1)=(-\delta-\alpha)(\beta-1)+1$.
		Moreover,
		$y^{[1]}\in\Pi(tu'(t))=\Pi(tp(t)\Phi(y(\tau(t))))$ and
		$y\in\IS_\infty$ since the index $\varrho$ of regular variation is
		positive. We have shown that $\IS\subset\NRV(\varrho)$. Next we
		derive asymptotic formulae by applying again Theorem~\ref{T:SV}.

If		$\int^\infty(s\widetilde p(s)\widetilde r(s))^{\alpha-1}\drm
		s=\infty$, then
		\begin{equation} \label{uf1}
		u(t)=\exp\left\{\int_a^t(1+o(1))\Phi\left(\frac{s\widetilde
			p(s)}{(\widetilde\delta+1)\widetilde r(s)}\right)\,\drm s\right\}
		\end{equation}
		as $t\to\infty$ and $u\in\IS_\infty^{\eqref{re}}$.

If
		$\int^\infty(s\widetilde p(s)\widetilde r(s))^{\alpha-1}\drm
		s<\infty$, then
		\begin{equation} \label{uf2}
		u(t)=M\exp\left\{-\int_t^\infty(1+o(1))\Phi\left(\frac{s\widetilde
			p(s)}{(\widetilde\delta+1)\widetilde r(s)}\right)\,\drm s\right\}
		\end{equation}
		as $t\to\infty$ and $u\in\IS_B^{\eqref{re}}$, where
		$M=\lim_{t\to\infty}u(t)\in(0,\infty)$.

Since
		$$
		\left(\frac{t\widetilde p(t)}{\widetilde
			r(t)}\right)^{\alpha-1}=H_\tau(t)\ \ \text{and}\ \
		\widetilde\delta+1=\varrho,
		$$
		from \eqref{uf1} and $y'(t)\sim\varrho y(t)/t$, under the assumption $\int_a^\infty H_\tau=\infty$, we get
		\begin{equation} \label{yf1}
		 y(t)=(1+o(1))\frac{1}{\varrho}tr^{1-\beta}(t)\exp\left\{\int_a^t(1+o(1))\frac{\beta-1}{\varrho^{\alpha-1}}
		H_\tau(s)\,\drm s\right\}
		\end{equation}
		as $t\to\infty$. Because of divergence of $\int_a^\infty H_\tau$,
		the expression $(1+o(1))/{\varrho}$ can be included into
		the $(1+o(1))$ term inside the integral in \eqref{yf1} (similarly as in \eqref{lny}), and thus
		we obtain formula \eqref{f11}. Since $y\in\IS_\infty$ and
		$u\in\IS_\infty^{\eqref{re}}$, we have $y\in\IS_{\infty,\infty}$.
		
Now assume that $\int_a^\infty H_\tau<\infty$. Then \eqref{uf2}
		implies
		$$
		 y'(t)=r^{1-\beta}(t)M^{\beta-1}\exp\left\{-\int_t^\infty(1+o(1))\frac{\beta-1}{\varrho^{\alpha-1}}H_\tau(s)\,\drm s\right\}
		$$
		as $t\to\infty$. Integrating from $b$ to $t$, replacing the lower
		limit in the integral by $a$ and $y(b)$ by a suitable $A$, we
		obtain formula \eqref{f21}. Moreover $y\in\IS_{\infty,B}$, in view
		of $y\in\IS_\infty$ and $u\in\IS_B^{\eqref{re}}$. From
		Theorem~\ref{T:SV} we know that
		$$
		|M-u|\in\SV\ \ \text{and}\ \ \frac{L_{\widetilde p}^{\alpha-1}(t)}{L_{\widetilde r}^{\alpha-1}(M-u(t))}=o(1)
		$$
		as $t\to\infty$. Noting that, by Lemma~\ref{L:fg},
		$$
		\frac{L_{\widetilde p}^{\alpha-1}(t)}{L_{\widetilde r}^{\alpha-1}}\sim\left(\frac{t}{\tau(t)}\right)^{\delta+\alpha}
		(\tau'(t))^{\alpha-1}\frac{L_p(t)}{L_r(t)}
		$$
		as $t\to\infty$, we get $|M-y^{[1]}|\in\SV$ and \eqref{LLM} easily follow.
		The equalities among the subclasses follow from the relations		 $\IS\subseteq\mS_{\NRV}(\varrho)\subseteq\mS_{\RV}(\varrho)\subseteq\IS\subseteq\IS_{\infty,x}\subseteq\IS$,
		$x=B$ or $x=\infty$ according to whether $\int_a^\infty H_\tau$ converges or diverges, respectively; note a regularly varying solution of
		\eqref{e} with a positive index is necessarily increasing.
	\end{proof}

	\begin{proof}[Proof of Theorem~\ref{T:gen1}]
		First note that, in view of
		Lemma~\ref{L:tau}-(ii), \eqref{tau} implies \eqref{tau1} with
		$i=D$. Take $y\in\IS$. Then $x(s)=y(t)$, with $s=R_D(t)$,
		satisfies equation \eqref{te}, where $\widehat r=1$, $\widehat
		p=p_D$, $\widehat\tau=\tau_D$, and $x\in\IS^{\eqref{te}}$, see
		Lemma~\ref{L:transf}. Note that the interval $[a,\infty)$ is
		transformed into $[b,\infty)$ for some $b$. Since $\widehat
		r\in\SV$ and $\widehat p\in\RV(-\alpha)$, we can apply
		Theorem~\ref{T:RV} on equation \eqref{te}; in fact, here
		$\delta=-\alpha<-1$. We get $x\in\NRV(1)$ since $\delta=-\alpha$ implies $\varrho=1$.
		Moreover,
		\begin{equation} \label{u11}
		x(s)=s\exp\left\{\int_b^s(1+o(1))(\beta-1)(u\widehat\tau'(u))^{\alpha-1}\widehat
		p(u)\,\drm u\right\}
		\end{equation}
		as $s\to\infty$, with $x\in\IS_{\infty,\infty}^{\eqref{te}}$
		provided $\int_b^\infty (u\widehat\tau'(u))^{\alpha-1}\widehat
		p(u)\,\drm u=\infty$, while
		\begin{equation} \label{u22}
		x(s)=B+\int_b^s
		M^{\beta-1}\exp\left\{-\int_\xi^\infty(1+o(1))(\beta-1)(u\widehat\tau'(u))^{\alpha-1}\widehat
		p(u)\,\drm u\right\}\,\drm\xi
		\end{equation}
		as $s\to\infty$, for some $B\in\R$, with
		$x\in\IS_{\infty,B}^{\eqref{te}}$ provided $\int_b^\infty
		(u\widehat\tau'(u))^{\alpha-1}\widehat p(u)\,\drm u=\infty$. Since
		$\widehat p=p_D$ and $\widehat\tau=\tau_D$, substitutions
		$v=R^{-1}_D(u)$ and $\eta=R^{-1}_D(\xi)$ in the integrals \eqref{u11} and \eqref{u22} yield formulae
		\eqref{tF11} and \eqref{tF22}. The fact that $y\in\IS_{\infty,i}$,
		$i\in\{B,\infty\}$, follows from $x\in\IS_{\infty,i}^{\eqref{te}}$
		and \eqref{x1y1}, recalling that $x(R^{-1}_D(s))=y(t)$, $\widehat
		p=p_D$, and $\widehat \tau=\tau_D$. From $|M-x^{[1]}|\in\SV$ we obtain $|M-y^{[1]}\circ R_D^{-1}|\in\SV$. In view of Lemma~\ref{L:tau} and condition \eqref{tau2}, we have
		$\tau_D(s)\asymp s$ as $s\to\infty$. Hence, expressing relation \eqref{LLM}
		in terms of our setting, we obtain
		$$
		\frac{(\widehat\tau'(s))^{\alpha-1}L_{\widehat
				p}(s)}{M-x^{[1]}(s)}=o(1)
		$$
		as $s\to\infty$. Recalling that $L_{\widehat{p}}(s)=s^\alpha\widehat p (s)=R_D^\alpha(t)p(t)r^{\beta-1}(t)$, we get \eqref{trNo2}.
	\end{proof}
	
	\begin{proof}[Proof of Theorem~\ref{T:gen2}]
		The proof
		is similar to that of Theorem~\ref{T:gen1}. Here, for $y\in\mS$,
		we make the substitution $x(s)=y(t)$ with $s=Q(t)$. In view of
		Lemma~\ref{L:transf}, $x$ satisfies equation \eqref{te}, where
		$\widehat r=s^{2\alpha-2}\in\RV(2\alpha-2)$, $\widehat
		p=p_C\in\RV(\alpha-2)$, and $\widehat\tau=\tau_C$. From
		Lemma~\ref{L:tau}-(ii), conditions \eqref{tau1} with $i=C$ holds. Since, in
		fact, $\delta=\alpha-2>-1$, we can apply Theorem~\ref{T:SV} to equation \eqref{te} where
		we consider $x\in\IS^{\eqref{te}}$ (i.e., $y\in\IS$). The
		details are left to the reader.
	\end{proof}

%%============================================================================%%
%%                                References                                  %%
%%============================================================================%%

\end{document}